\documentclass[letterpaper, 10 pt, conference]{ieeeconf}  

\IEEEoverridecommandlockouts                              
\usepackage[colorlinks=true, pdfstartview=FitV, linkcolor=black, citecolor= black, urlcolor= black]{hyperref}
\usepackage[left=20mm, right=20mm, top=20mm, bottom=20mm]{geometry} 
\usepackage[english]{babel}
\usepackage{graphicx,subcaption}
\usepackage[nosumlimits,nonamelimits]{amsmath}
\usepackage{amssymb}
\usepackage{amsthm}
\usepackage[bb=libus]{mathalpha}
\usepackage{cases}
\usepackage{bm}
\usepackage{dsfont}
\usepackage{breqn}
\usepackage{physics}
\usepackage{xr}
\usepackage{algorithm,algpseudocodex}
\usepackage{cite}

\newcommand{\R}{\mathbb{R}}

\newcommand{\diff}{\mathrm{d}}

\let\P\relax
\newcommand{\P}[1]{\mathbb{P}[#1]}
\newcommand{\E}[1]{\mathbb{E}[#1]}

\newcommand{\suchthat}{\mathrm{s.t.}\quad}
\newcommand{\normal}{\mathcal{N}}
\newcommand{\indicator}{\mathds{1}}
\newcommand{\chol}[1]{(#1)^{\frac{1}{2}}}
\newcommand{\maxeig}[1]{\lambda_{\max} \left\{#1 \right\}}
\newcommand{\blkdiag}[1]{\mathrm{blkdiag}\left(#1\right)}
\newcommand{\diag}[1]{\mathrm{diag}(#1)}

\usepackage{cleveref}
\crefname{align}{}{}
\crefname{equation}{}{}
\crefname{figure}{Fig.}{Figs.}
\crefname{table}{Table}{Tables}
\crefname{theorem}{Theorem}{Theorems}
\crefname{definition}{Definition}{Definitions}
\crefname{lemma}{Lemma}{Lemmas}
\crefname{remark}{Remark}{Remarks}
\crefname{assumption}{Assumption}{Assumptions}
\crefname{proof}{Proof}{Proofs}
\crefname{algorithm}{Algorithm}{Algorithms}
\crefname{problem}{Problem}{Problems}
\crefname{proposition}{Proposition}{Propositions}
\crefname{corollary}{Corollary}{Corollaries}
\crefname{section}{Section}{Sections}

\newtheorem{problem}{Problem}
\newtheorem{proposition}{Proposition}
\newtheorem{remark}{Remark}
\newtheorem{theorem}{Theorem}

\def\shortOrFull{2} 

\def\showChanges{0} 

\newcommand{\highlight}[1]{%
    \ifnum\showChanges=1
        \textcolor{red}{#1}%
    \else
        #1%
    \fi
}
\newcommand{\mathhl}[1]{
    \ifnum\showChanges=1
        \mathcolor{red}{#1}
    \else
        #1
    \fi
}

\allowdisplaybreaks[1] 

\title{\LARGE \bf
Chance-Constrained Gaussian Mixture Steering to a Terminal Gaussian Distribution
}

\author{Naoya Kumagai\thanks{This work was supported by the U.S. Air Force Office of Scientific Research through research grant FA9550-23-1-0512. N. Kumagai acknowledges support for his graduate studies from the Shigeta Education Fund. Naoya Kumagai and Kenshiro Oguri are with the School of Aeronautics and Astronautics, Purdue University, West Lafayette, Indiana, 47907, USA. Emails: 
    {\tt\small nkumagai@purdue.edu},
    {\tt\small koguri@purdue.edu}
    }  and Kenshiro Oguri}

\begin{document}

\maketitle
\thispagestyle{empty}
\pagestyle{empty}

\begin{abstract}
We address the problem of finite-horizon control of a discrete-time linear system, where the initial state distribution follows a Gaussian mixture model, the terminal state must follow a specified Gaussian distribution, and the state and control inputs must obey chance constraints. We show that, throughout the time horizon, the state and control distributions are fully characterized by Gaussian mixtures. We then formulate the cost, distributional terminal constraint, and affine/2-norm chance constraints on the state and control, as convex functions of the decision variables. This is leveraged to formulate the chance-constrained path planning problem as a single \highlight{convex optimization} problem. 
A numerical example demonstrates the effectiveness of the proposed method. 
\end{abstract}

\section{Introduction}

In this paper, we address the problem of finite-horizon control of a discrete-time linear system with an initial state distributed with a Gaussian mixture model (GM, GMM). The task is to steer the distribution to a terminal Gaussian distribution while obeying affine and 2-norm chance constraints on the state and control input. 

There has been extensive work on the linear covariance steering problem \cite{hotz_covariance_1987,chen_optimal_2016, okamoto_optimal_2018,liu_optimal_2023}. The covariance steering problem was first introduced by Hotz and Skelton \cite{hotz_covariance_1987}. More recently, Chen et al. derived the solution to the continuous linear covariance steering problem\cite{chen_optimal_2016}. The discrete linear covariance steering problem under affine chance constraints was shown to be formulated into a semidefinite programming problem \cite{okamoto_optimal_2018}. The existence and uniqueness of the optimal control law for a quadratic cost function and a change of variables for efficient computation was demonstrated in \cite{liu_optimal_2023}. 

The problem of general distribution steering has also been investigated. \cite{sivaramakrishnan_distribution_2022} uses characteristic functions (CFs) to address the most general case under linear dynamics: the initial and final distributions, as well as the process noise distribution, are all arbitrary. While offering a powerful general framework, the algorithm utilizes potentially computationally expensive tools such as nonconvex programming and quadrature for CF inversion. 
Methods in optimal transport theory \cite{chen_optimal_2016,caluya_reflected_2021} \highlight{allow} steering from arbitrary initial to final distributions. \cite{caluya_reflected_2021} extends the theory of Schrodinger bridges to account for path constraints; however, since their method relies on a Fokker Planck initial boundary value problem solver as an internal workhorse, the accuracy and reliability of the algorithm are unclear. 

Several works focus on a specific case of the problem where the initial or final distributions have a Gaussian mixture model. 
\cite{boone_non-gaussian_2022} formulates a problem of controlling a Gaussian mixture under chance constraints via a one-instance control input (as opposed to control over the entire horizon). The solution method is via nonlinear programming and thus does not provide a theoretical guarantee on the solution quality or convergence. It also applies risk allocation \cite{ono_iterative_2008} to allocate the risk between each Gaussian kernel and chance constraint.
\cite{hu_chance_2022} proposes a branch-and-bound algorithm for finding the globally optimal risk allocation between Gaussian kernels. \cite{yang_analytical_2020} models wind power uncertainty via GMMs and solves a chance-constrained unit commitment problem. \cite{ren_chance-constrained_2023} considers chance constraints under GMM uncertainty for trajectory planning of autonomous vehicles, where the uncertainty arises from the movement of other vehicles. \cite{balci_density_2023} proposes a random control policy that steers an initial Gaussian mixture to a final Gaussian mixture under deterministic linear dynamics. The solution is obtained via linear programming. 

A problem of interest yet to be addressed is the steering of an initial GM distribution to a terminal Gaussian distribution subject to state and control input chance constraints. 
The choice of this problem is motivated by common engineering scenarios; for example, when the task is to steer the distribution of an autonomous vehicle, a high-order uncertainty quantification algorithm may provide a non-Gaussian initial distribution, which can be fitted with a Gaussian mixture \cite[Ch. 3.2]{stergiopoulos_advanced_2017},\cite{dempster_maximum_1977}.  
We show that by choosing a control policy with a deterministic feedforward term and a probabilistic feedback sequence which is probabilistically chosen after observation of the initial state, the problem is formulated into a \highlight{convex optimization} problem. \highlight{As such, this work expands the set of problem settings in the distribution steering literature that can be formulated as a convex optimization.}

The contributions of this study are threefold. First, we show that using the proposed control policy, throughout the time horizon, the state and control can be fully characterized by Gaussian mixture distributions with analytical mean and covariance for each Gaussian kernel. Second, based on this result, we formulate deterministic convex formulations of 1) the terminal distributional constraint such that the final state follows a Gaussian distribution and 2) affine and 2-norm chance constraints on the state and control variables.
Third, we outline a modified version of the iterative risk allocation algorithm which reduces conservativeness in the risk allocation and hence achieves better optimality.

\section{Preliminaries} \label{sec:preliminaries}
\textit{Notation:}
$I$ represents the identity matrix of appropriate size. 
A random vector $x$ with normal distribution of mean $\mu$ and covariance matrix $\Sigma$ is denoted as $x\sim\normal(\mu, \Sigma)$. 
\highlight{
 When $x$ is GM-distributed with weights $\mathhl{(\alpha_i)_{i=1:K}}$ such that $\sum_{i=1}^K \alpha_i = 1)$, means $\mathhl{(\mu_i)_{i=1:K}}$, and covariance matrices $(\Sigma_i)_{i=1:K}$, it is denoted as $x\sim\mathrm{GMM}(\alpha_i, \mu_i, \Sigma_i)_{i=1:K}$.%
 }
The probability density function (PDF) of a vector $x$ evaluated at $\hat{x}$ is denoted as $f_x(\hat{x}$). 
If $x\sim \normal(\mu, \Sigma)$, we may write $f_x(\hat{x}) = f_{\normal}(\hat{x}; \mu, \Sigma)$ for emphasis. 
For random variables $x$ and $y$, $x|y=\hat{y}$ denotes the random variable $x$ conditioned on $y=\hat{y}$; for ease of notation, we may also use $x|\hat{y}$. 
The conditional density function of $x$ given $y=\hat{y}$ is written as $f_{x|\hat{y}}$. 
For a symmetric matrix $\Sigma$, we write $\Sigma \succ 0 (\succeq 0)$ if $\Sigma$ is positive (semi-)definite. 
For a matrix $X$, $X^{\frac{1}{2}}$ refers to the matrix such that $X^{\frac{1}{2}} \mathhl{(X^{\frac{1}{2}})}^\top = X$.
\highlight{
$\P{\cdot}, \E{\cdot}, \maxeig{\cdot}, \norm{\cdot}$ calculate probability, expected value, largest eigenvalue, and 2-norm respectively.
}
\highlight{$\bigwedge_k A_k$ denotes the intersection of the events $A_k$ for all $k$.}
$\mathhl{\mathbb{i} = \sqrt{-1}}.$

\subsection{Problem Formulation}
Consider the discrete, linear time-varying system:
\begin{equation}
    x_{k+1} = A_k x_k + B_k u_k, \quad k=0,1,\cdots,N-1
\end{equation}
with state $x_k \in \R^{n_x}$, control input $u_k \in \R^{n_u}$, and matrices $A_k, B_k$ of appropriate sizes. We assume that the initial state is distributed such that $x_0 \sim \mathrm{GMM}\mathhl{(\alpha_i, \mu_0^i, \Sigma_0^i)_{i=1:K}}$ with $K$ kernels, and seek to steer the final state $x_N$ to the desired Gaussian distribution $x_N \sim \normal(\mu_f, \Sigma_f)$.  

This paper is concerned with the following stochastic optimal control problem:
\begin{problem} \label{pr:1}
\begin{subequations} \label{eq:pr1}
\begin{align}
    &\min_{\mathhl{\pi \in \Pi}} \ 
    J = \E{\mathcal{L}(x_0, x_1, \cdots, u_0, u_1, \cdots)} \label{eq:cost-original}\\
    &\suchthat x_0 \sim \mathrm{GMM}\mathhl{(\alpha_i, \mu_0^i, \Sigma_0^i)_{i=1:K}}\\
    &x_N \sim \mathcal{N}(\mu_f, \Sigma_f) \label{eq:terminal-dist}\\
    & x_{k+1} = A_k x_k + B_k u_k\\
    &\hspace{-3pt}\P{\bigwedge_{k=0}^{N} x_k \in \mathcal{X}} \geq 1 - \Delta ,  \  \P{\bigwedge_{k=0}^{N - 1} u_k \in \mathcal{U}} \geq 1 - \Gamma  \label{eq:CC-original}%
\end{align}%
\end{subequations}%
\end{problem}%
$\mathhl{\pi = (\pi_k(\cdot))_{1:N-1}}$ \highlight{is the control policy, which calculates the input as} $\mathhl{u_k = \pi_k(x_0, \Omega)}$\highlight{, with $\Omega$ being the set of parameters of $\pi$. $\Pi$ is the set of all admissible policies.} 
$\mathcal{L}$ is the cost function, defined below. $\mathhl{\Sigma_f \succ 0.}$
$\Delta, \Gamma$ represent the allowed probability of violation for each constraint and satisfy $0 \leq \Delta, \Gamma \leq 0.5$. This is a reasonable assumption since most chance-constrained problems allow a risk smaller than 0.5. 
In this work, we address two types of cost functions, 
\begin{numcases}{\mathcal{L} = }
    \sum\nolimits_{k=0}^{N-1}   x_k^\top Q_k x_k + u_k^\top R_k u_k \ (Q_k \succeq 0, R_k \succ 0) \label{eq:lqr}\\
     \sum\nolimits_{k=0}^{N-1} \norm{u_k}  \label{eq:minfuel}
\end{numcases}
respectively termed the quadratic cost and the 2-norm cost.

Consider a concatenated formulation \cite{goldshtein_finite-horizon_2017} such that
\begin{equation} \nonumber
    \mathhl{
    X {=} \smqty[x_0 \\ x_1 \\ \vdots ],\  
    U {=} \smqty[u_0 \\ u_1 \\ \vdots], \ 
    }
    A {=} \smqty[I \\ A_0 \\ A_1 A_0 \\ \vdots], \  
    B {=} \smqty[
        0 & 0 & \\
        B_0 & 0 & \\
        A_1 B_0 & B_1 & \\
        \vdots & & \ddots
    ]
\end{equation}
Then, the state process can be written as $X = A x_0 + B U $. 
Define $E_k$ to be the sparse matrix such that $x_k = E_k X$.



%
\section{GM Steering without Chance Constraints} \label{sec:steering}
We choose the affine feedback control policy
\begin{equation} \label{eq:policy}
    u_{\mathhl{k,i}}(x_0) = v_k + L_k^i (x_0 - \mu_0^g)
\end{equation}
where $\mu_0^g \triangleq \E{x_0} = \sum_{i=1}^{K} \alpha_i \mu_0^i$, and \highlight{the feedback gains of the control sequence are determined by sampling from the discrete distribution}
\begin{equation}
    \mathhl{\P{L_k = L_k^i, \forall k} = \lambda_i(x_0),}
\end{equation}
\begin{equation}
    \text{\highlight{where}} \quad \lambda_i(x_0) = \frac{\alpha_i f_{\normal}(x_0; \mu_0^i, \Sigma_0^i)}{\sum_{i=1}^{K} \alpha_i f_{\normal}(x_0; \mu_0^i, \Sigma_0^i)}
\end{equation}
This control policy is inspired from \cite{balci_density_2023}; however, a key difference is that although the feedback term is chosen probabilistically, the feedforward trajectory is \highlight{deterministic and} common to all possible control policies. \highlight{In \cref{sec:discussion}, we elaborate on the difference between our work and \cite{balci_density_2023}.}
Define $L^i = [(L_0^i)^\top, \cdots, (L_{N-1}^i)^\top]^\top, V = [v_0^\top, \cdots, v_{N-1}^\top]^\top, R = \blkdiag{R_0, \cdots, R_{N-1}}, Q = \blkdiag{Q_0, \cdots, Q_{N-1}}$.

\subsection{State distribution}
First, we derive the distribution of the state under the proposed policy. 
We note two important propositions involving CFs that assist us in the proof. 
\begin{proposition} \label{prop:CF_normal}
    Let $X \sim \normal(\mu, \Sigma)$. Its CF is
    \begin{equation}
        \phi_X(t) = \mathhl{\exp}\big(\mathhl{\mathbb{i}} t^\top \mu - \frac{1}{2}t^\top \Sigma t)
    \end{equation}
\end{proposition}
\begin{proof}
    This can be found in most probability textbooks.
\end{proof}

\begin{proposition} \label{prop:implicit_normal_dist}
    Let $X \in \R^p, X \sim \normal(\mu, \Sigma)$. Then,
    \begin{equation}
        \int_{\R^p} f_X(x)\  \delta (y - (Dx+d)) \ \diff x = f_Y(y) \label{eq:prop1}
    \end{equation}
    where  $D\in\R^{q\times p}, d \in \R^q, Y \sim \normal(D\mu+d, D\Sigma D^\top)$.
\end{proposition}

\begin{proof}
    Let $Z \in \R^q$ be the random variable that has the PDF of the LHS of \cref{eq:prop1}. The CF of $Z$ is defined by $\phi_Z(t) = \E{\mathhl{\exp}\big(\mathhl{\mathbb{i}}t^\top Z)}$
    where $t \in \R^q$ is a deterministic vector. Then,
    \if\shortOrFull1 
    \begin{align} \nonumber
        \phi_Z(t) &= \int_{\R^q} \mathhl{\exp}\big(\mathhl{\mathbb{i}}t^\top z) \int_{\R^p} f_X(x) \delta (z - (Dx+d)) \ \diff x\diff z \nonumber\\
        &= \int_{\R^p} f_X(x) \int_{\R^q}  \mathhl{\exp}\big(\mathhl{\mathbb{i}}t^\top z) \delta (z - (Dx+d)) \ \diff z \diff x \nonumber\\
        &= \int_{\R^p} f_X(x) \mathhl{\exp}\big(\mathhl{\mathbb{i}}t^\top(Dx+d)) \ \diff x \nonumber\\
        &= \mathhl{\exp}\big(\mathhl{\mathbb{i}}t^\top d) \E{\mathhl{\exp}\big(\mathhl{\mathbb{i}}s^\top Z)} \nonumber \\
        &= \mathhl{\exp}\big(\mathhl{\mathbb{i}}t^\top d) \mathhl{\exp}\big(\mathhl{\mathbb{i}} s^\top \mu - \frac{1}{2}s^\top \Sigma s) \nonumber
    \end{align}
    \fi
    \if\shortOrFull2 
    \begin{align*}
        \phi_Z(t) &= \int_{\R^q} \mathhl{\exp}\big(\mathhl{\mathbb{i}}t^\top z) \int_{\R^p} f_X(x) \delta (z - (Dx+d)) \ \diff x\diff z \nonumber\\
        &= \int_{\R^p} f_X(x) \int_{\R^q}  \mathhl{\exp}\big(\mathhl{\mathbb{i}}t^\top z) \delta (z - (Dx+d)) \ \diff z \diff x \nonumber\\
        &= \int_{\R^p} f_X(x) \mathhl{\exp}\big(\mathhl{\mathbb{i}}t^\top(Dx+d)) \ \diff x \nonumber\\
        &= \mathhl{\exp}\big(\mathhl{\mathbb{i}}t^\top d) \E{\mathhl{\exp}\big(\mathhl{\mathbb{i}}s^\top Z)} \nonumber \\
        &= \mathhl{\exp}\big(\mathhl{\mathbb{i}}t^\top d) \mathhl{\exp}\big(\mathhl{\mathbb{i}} s^\top \mu - \frac{1}{2}s^\top \Sigma s) \nonumber \\
        &=\exp\big(\mathbb{i}t^\top d + \mathbb{i} t^\top D \mu - \frac{1}{2} t^\top D \Sigma D^\top t) 
    \end{align*}
    \fi    
    where $s = D^\top t$.
    The first equality is from the definition of expectation; the third equality is from the sifting property of $\delta(\cdot)$.
    \if\shortOrFull1
    \highlight{The final equality is from \cref{prop:CF_normal}.}
    Further manipulating, we get 
    \fi
    \if\shortOrFull2
    The fifth equality is from \cref{prop:CF_normal}.
    Finally, we have
    \fi
    \begin{align}
        \hspace{-8pt}\phi_Z(t) = \mathhl{\exp}(\mathhl{\mathbb{i}}t^\top(D\mu+d) - \frac{1}{2} t^\top D \Sigma D^\top t) = \phi_Y(t)
    \end{align}
    By the Inversion Theorem of CFs \cite{grimmett_probability_2020}, $Z$ and $Y$ have the same distribution function. 
\end{proof}

\begin{remark}
    \cref{prop:implicit_normal_dist} makes no assumptions about the invertibility of the matrix $D$. If we assume the invertibility of $D$, we can simply invert the expression inside the $\delta$ of \cref{eq:prop1} and use the sifting property of $\delta(\cdot)$ to obtain the same result. The method of proof shown here bypasses any restrictions on invertibility. For example, \cite[Proposition 3]{balci_density_2023} appears to have this assumption implicitly.
\end{remark}

\begin{proposition} \label{prop:x_k-dist}
    Under the proposed control policy, the state is GM-distributed throughout the time horizon, i.e. $x_k \sim \text{GMM} \mathhl{(\alpha_i, \mu_k^i, \Sigma_k^i)_{i=1:K}}$, where 
    \begin{align}
        \mu_k^i 
        &= E_k [A \mu_0^i + B V  + B L^i ( \mu_0^i - \mu_0^g) ] \label{eq:mean-traj}\\
        \Sigma_k^i 
        &= E_k (A + B L^i) \Sigma_0^i (A + B L^i)^\top E_k^\top
    \end{align}
\end{proposition}

\begin{proof}
Using the definition of conditional probability densities and marginal densities, 
\begin{equation}
\hspace{-5pt}
    f_{X}(\hat{X}) = \int \int \ 
    f_{X | \hat{x}_0, \hat{U}}(\hat{x})\ 
    f_{U | \hat{x}_0} (\hat{U}) 
    f_{x_0} (\hat{x}_0) \ 
    \diff \hat{U} \ 
    \diff \hat{x}_0
\end{equation}
where we have expressions for the following conditional densities: 
\begin{align}
    &\hspace{-5pt} f_{X | \hat{x}_0, \hat{U}}(\hat{X}) = \delta(\hat{X} - (A \hat{x}_0 + B \hat{U})) \\
    &\hspace{-5pt} f_{U | \hat{x}_0} (\hat{U}) = \sum\nolimits_{i=1}^{K} \lambda_i (\hat{x}_0) \delta ( \hat{U} - (L^i(\hat{x}_0 - \mu_0^g) + V))
\end{align}
The rest of the proof is similar to that of \cite[Proposition 3]{balci_density_2023}, and is therefore omitted. The proof utilizes the Dirac delta function $\delta(\cdot)$ to express discrete distributions over a continuous support \cite{chakraborty_applications_2008}, as well as \cref{prop:implicit_normal_dist}. 
\end{proof}

\subsection{Control distribution}
\begin{proposition} \label{prop:control-distribution}
Under the proposed policy, $u_k$ is also GM-distributed with $u_k \sim \mathrm{GMM} \mathhl{(\alpha_i, \mu_{u,k}^i, \Sigma_{u,k}^i)_{i=1:K}}$, where
\begin{align}
    \mu_{u,k}^i &= v_k + L_k^i (\mu_0^i - \mu_0^g), \quad 
    \Sigma_{u,k}^i = L_k^i \Sigma_0^i ({L_k^i})^\top
\end{align}
\end{proposition}
\begin{proof}
Using the definition of conditional probability densities and marginal densities, the distribution of $U$ is
\begin{align*} 
    &f_U(\hat{U}) = \int f_{U|x_0 = \hat{x}_0} (\hat{U}) f_{x_0}(\hat{x}_0) \ \diff \hat{x}_0 \\
    & \hspace{-5pt} = \int f_{x_0}(\hat{x}_0) \sum\nolimits_{i=1}^{K} \lambda_i (\hat{x}_0) \delta ( \hat{U} - (L^i(\hat{x}_0 - \mu_0^g) + V)) \diff \hat{x}_0 \\
    &\hspace{-5pt} =\sum_{i=1}^{K} \alpha_i \int f_{\normal} (\hat{x}_0; \mu_0^i, \Sigma_0^i) \delta ( \hat{U} - (L^i(\hat{x}_0 - \mu_0^g) + V)) \diff \hat{x}_0 \\
    &\hspace{-5pt} = \sum\nolimits_{i=1}^{K} \alpha_i f_{\normal} (\hat{U}; V + L^i (\mu_0^i - \mu_0^g)  , L^i \Sigma_0^i (L^i)^\top)\tag*{\qedhere} 
\end{align*}%
\end{proof}%
Note, the mean of the control $u_k$ is \textit{not} the feedforward term $v_k$, but is shifted by $\sum_{i=1}^{K} \alpha_i L_k^i (\mu_0^i - \mu_0^g)$.

\subsection{Terminal constraint sufficient condition}
Next, we present a sufficient condition for \cref{eq:terminal-dist}. 
\begin{proposition} \label{prop:terminal}
The terminal constraint can be satisfied by the constraints
\begin{equation}
     \mu_N^i = \mu_f \quad \forall i, \quad \Sigma_f = \sum\nolimits_{i=1}^{K} \alpha_i \Sigma_N^i
\end{equation}
\end{proposition}

\begin{proof}
A Gaussian distribution $\normal (\mu^g, \Sigma^g)$ which is expressed exactly with a Gaussian mixture model $\mathrm{GMM}\mathhl{(\alpha_i, \mu_i, \Sigma_i)_{i=1:K}}$ satisfies the following\cite{boone_non-gaussian_2022}: 
\begin{align}
    \hspace{-5pt}\mu^g = \sum_{i=1}^{K} \alpha_i \mu_i, \  \Sigma^g = \sum_{i=1}^{K} \alpha_i (\Sigma_i + \mu_i \mu_i^\top) - \mu^g (\mu^g)^\top
    \label{eq:GMM-Gaussian-cov}
\end{align}
Since each kernel belonging to the GM remains Gaussian under affine dynamics and control, a sufficient condition for the terminal constraint \cref{eq:terminal-dist} to be satisfied is $\mu_N^i = \mu_f \quad \forall i$. Substitute this into the second equation of \cref{eq:GMM-Gaussian-cov} to obtain
\begin{equation}
    \Sigma^g = \sum\nolimits_{i=1}^{K} \alpha_i \Sigma_i \tag*{\qedhere} 
\end{equation}
\end{proof}
We can relax the equality above with $\succeq$ for practical purposes \cite{bakolas_optimal_2016} so that the terminal state is concentrated within the target covariance ellipsoid.

\begin{proposition} \label{prop:terminal-numerical}
    The terminal constraint \cref{eq:terminal-dist} is implied by 
    \begin{align}
        &\mu_f = E_N [A \mu_0^i + B V + B L^{\mathhl{i}} ( \mu_0^i - \mu_0^g)] \quad \forall i \label{eq:terminal-mean-final} \\
        &\norm*{\Sigma_f^{-\frac{1}{2}} Y} \leq 1 \label{eq:terminal-cov-GMM-final}
    \end{align}
    where 
    \begin{align}
       &Y = [\sqrt{\alpha_1} \chol{\Sigma_N^{(1)}}, \sqrt{\alpha_2} \chol{\Sigma_N^{(2)}}, \cdots, \sqrt{\alpha_{K}} \chol{\Sigma_N^{(K)}} ] \label{eq:Y-def}\\
        &\chol{\Sigma_N^{i}} = E_N(A + B L^{\mathhl{i}}) \chol{\Sigma_0^i}
    \end{align}
\end{proposition}
\begin{proof}
From \cref{prop:terminal,prop:x_k-dist}, in order to satisfy the terminal constraint, we need \cref{eq:terminal-mean-final} and 
\begin{align}
    \Sigma_f &\succeq \sum\nolimits_{i=1}^{K} \alpha_i E_N(A + B L^{\mathhl{i}}) \Sigma_0^i (A + B L^{\mathhl{i}})^\top E_N^\top \label{eq:terminal-cov-GMM-ver2}
\end{align}
With a proof similar to the one used in \cite[Proposition 4]{okamoto_optimal_2018}, \cref{eq:terminal-cov-GMM-ver2} is equivalent to
\begin{equation} \label{eq:lambda_max<1}
    \maxeig{\Sigma_f^{-\frac{1}{2}} \left[\sum\nolimits_{i=1}^{K} \alpha_i \Sigma_N^{i} \right] (\Sigma_f^{-\frac{1}{2}})^{\top} } \leq 1
\end{equation}
Defining $Y$ in \cref{eq:Y-def}, this is equivalent to
\begin{equation}
    \maxeig{\Sigma_f^{-\frac{1}{2}} Y Y^{\top} (\Sigma_f^{-\frac{1}{2}})^\top} \leq 1 \ \Leftrightarrow\  \norm{\Sigma_f^{-\frac{1}{2}} Y}^2 \leq 1 \tag*{\qedhere} 
\end{equation}
\end{proof}

\subsection{Cost function}
Before formulating the cost function in terms of the decision variables, we note an important proposition to assist us in the proof. 
\begin{proposition} \label{prop:E-func}
    Let $\xi$ be a GM-distributed vector with $\xi \sim \mathrm{GMM}\mathhl{(\alpha_i, \mu_i, \Sigma_i)_{i=1:K}}$. Then, for \textit{any} function $H(\xi)$, 
    \begin{equation}
        \E{H(\xi)} = \sum\nolimits_{i=1}^{K} \alpha_i \E{H(\xi_i)}
    \end{equation}
    where $\xi_i \sim \normal(\mu_i, \Sigma_i)$. 
\end{proposition}
\begin{proof}
    See \cite{hu_chance_2022}. 
\end{proof}
\if\shortOrFull1 
\fi
\if\shortOrFull2 
Essentially, \cref{prop:E-func} states that the expectation of a function of a variable with a mixture distribution is equal to the expectation of the weighted sum of the expected values of the function with the individual distributions of the mixture as inputs. 
\fi
\begin{proposition} \label{prop:cost-deterministic}
The objective functions \cref{eq:lqr,eq:minfuel} can be equivalently expressed as
    
1) Quadratic cost
\vspace{-7pt}
\begin{multline}\label{eq:lqr-deterministic} 
    J(L^i,V) = \sum\nolimits_{i=1}^{K} \alpha_i \mathrm{tr}\{RL^i\Sigma_0^i (L^i)^\top 
    + Z^\top R Z\\
    + Q C \Sigma_0^i C^\top  
    + (A \mu_0^i + BZ)^\top Q (A \mu_0^i + BZ)\}
\end{multline}
where $Z \triangleq V + L^i (\mu_0^i - \mu_0^g),\ C \triangleq (A + B L^i)$.

2) 2-norm cost%
\vspace{-5pt}
\begin{equation}
    J(L^i,V) = \sum\nolimits_{k=0}^{N-1} \sum\nolimits_{i=1}^{K} \alpha_i \norm{v_k + L_k^i (\mu_0^i - \mu_0^g)} \label{eq:minfuel-deterministic}
\end{equation}
\end{proposition}

\begin{proof}
    We can apply standard algebraic manipulations to show this result based on \cref{prop:E-func} and the expression for the distribution of the state and control. 
\end{proof}

\begin{theorem} \label{thm:1}
    Without chance constraints \cref{eq:CC-original}, \cref{pr:1} can be converted to a convex optimization problem.
\end{theorem}
\begin{proof}
    \highlight{The terminal constraints involve only linear \cref{eq:terminal-mean-final} and norm \cref{eq:terminal-cov-GMM-final} operations on the variables, and the cost function is quadratic \cref{eq:lqr-deterministic} or a norm \cref{eq:minfuel-deterministic} of the variables.}
\end{proof}

\section{Chance Constraints} \label{sec:CCs}
In this section, we consider the deterministic reformulation of the chance constraints \cref{eq:CC-original} to show that under a fixed risk allocation, they can be formulated as deterministic convex constraints. 
\subsection{Deterministic Formulation}
We consider two types of chance constraints for a GM-distributed variable $y_k \in \R^{n_y}, y_k \sim \mathrm{GMM}\mathhl{(\alpha_i, \mu_i, \Sigma_i)_{i=1:K}}$, which can be applied to either state or control: 
\begin{equation}
    \P{ \bigwedge\nolimits_{j=1}^{N_c} \bigwedge\nolimits_{k=1}^{N} y_k \in S_j} \geq 1 - \Delta \label{eq:original-chance}
\end{equation}
where $S_j$ is a halfplane constraint, defined such that
\begin{equation}
    S_j(y_k) \triangleq \{y_k: a_j^\top y_k - b_j \leq 0\} 
\end{equation}
for $j = 1,\cdots, N_c$. We also consider the 2-norm constraint: 
\begin{equation}
    \P{\bigwedge\nolimits_{k=0}^{N-1} \norm{G y_k + g} \leq y_{\max}} \geq 1 - \Gamma \label{eq:control-norm}
\end{equation}

\begin{theorem} \label{thm:CC-deterministic}
    The chance constraints \cref{eq:original-chance,eq:control-norm} can be conservatively approximated in a deterministic form as:
    \begin{subequations} 
    \begin{align} 
    & \hspace{-5pt} a_j^\top \mu_k^i + \mathhl{F_{\normal}^{-1}}(1 - \delta_{ijk}) \norm{a_j^\top \chol{\Sigma_k^i}} \leq b_j  \quad \forall (i,j,k)  \label{eq:CC_deterministic-affine}\\
      &\hspace{-5pt}
      \sum\nolimits_{i=1}^{K}
      \sum\nolimits_{j=1}^{N_c} \sum\nolimits_{k=1}^N\alpha_i \delta_{ijk} \leq \Delta \label{eq:delta_ijk_sum} \\
    & \hspace{-5pt} \norm{G\mu_{k}^i+g} {+}  \sqrt{\mathhl{F_{\chi_{n_y}^2}^{-1}}(1 -\gamma_{ik})} \norm{G\chol{\Sigma_{k}^i}} \leq y_{\max}, \forall (i,k) \label{eq:control-norm-final}\\
    &\hspace{-5pt}
    \sum\nolimits_{i=1}^{K}
    \sum\nolimits_{k=0}^{N-1}\alpha_i \gamma_{ik} \leq \Gamma \label{eq:gamma_ik_sum}
    \end{align}
    \end{subequations}
    where $\mathhl{F_{\normal}^{-1}}(\cdot)$ is the inverse cumulative distribution function (icdf) for a standard normal distribution, and $\mathhl{F_{\chi_{n_y}^2}^{-1}}(\cdot)$ is the icdf for a chi-squared distribution with $n_y$ degrees of freedom.
$\delta_{ijk}$ and $\gamma_{ik}$ are additional variables that indicate the risk allocated to each decomposed constraint. 
\end{theorem}
\begin{proof}
    Let $y_k^i$ be the normal-distributed vector such that $y_k^i \sim \normal(\mu_k^i, \Sigma_k^i)$. 
    Using Boole's inequality and introducing the decision variables $\delta_{jk}$, \cref{eq:original-chance} can be \textit{conservatively approximated as} \cite{blackmore_convex_2009}
    \begin{subequations} \label{eq:CC-after-Boole}
    \begin{align}
        &\P{y_k \in S_j} \geq 1 - \delta_{jk} \quad \forall (j,k) \label{eq:delta_jk}\\
        &\sum\nolimits_{j=1}^{N_c} \sum\nolimits_{k=1}^N\delta_{jk} \leq \Delta
    \end{align}
        \end{subequations}
    Now, introducing the decision variables $\delta_{ijk}$, \cref{eq:CC-after-Boole} can be \textit{equivalently} expressed as \cite{hu_chance_2022}
    \begin{subequations}\label{eq:CC-after-Boole-decomposed}
    \begin{align} 
        &\P{y_k^i \in S_j} \geq 1 - \delta_{ijk} \quad \forall(i,j,k) \label{eq:normal-constraint}\\
        &\sum\nolimits_{i=1}^{K} \alpha_i \delta_{ijk} \leq \delta_{jk} \quad \forall(j,k) \label{eq:delta_ijk}
    \end{align}
    \end{subequations}
Since $S_j$ is a hyperplane constraint, \cref{eq:normal-constraint} can be expressed as \cref{eq:CC_deterministic-affine} \cite{blackmore_convex_2009}. 
Combining \cref{eq:delta_jk,eq:delta_ijk}, we get \cref{eq:delta_ijk_sum}.

Similarly, for the 2-norm constraint, \cref{eq:original-chance} is conservatively approximated by introducing the decision variables $\gamma_{ik}$. Note that deriving \cref{eq:CC-after-Boole-decomposed} from \cref{eq:CC-after-Boole} is a general property that applies regardless of the constraint or mixture type. Whence,
    \begin{align}
        &\P{\norm{G y_k^i + g} \leq y_{\max}} \geq 1 - \gamma_{ik} \quad \forall(i,k) \label{eq:control-norm-ik} \\
        & \sum\nolimits_{i=1}^{K} \sum\nolimits_{k=0}^{N-1}\alpha_i \gamma_{ik} \leq \Gamma \label{eq:gamma_ik}
    \end{align}
    Using the triangle inequality and Boole's inequality \cite{oguri_chance-constrained_2024}, \cref{eq:control-norm-ik} is conservatively approximated by \cref{eq:control-norm-final}.
\end{proof}

\begin{remark}
  Contrary to the statement in \cite[p.3594]{boone_non-gaussian_2022}, the decomposition of the constraint on the mixture does not require Boole's inequality; it is equivalent since it is merely an introduction of additional decision variables. Boole's inequality is only used when considering risk allocation between chance constraints or nodes.
\end{remark}

\begin{remark}
The result of \cref{thm:CC-deterministic} is convex in the variables $V, L^i$ if we fix the risk variables $\delta_{ijk},\gamma_{ik}$. A simple solution is uniform risk allocation \highlight{(URA)}, which is easily derived as $\delta_{ijk} = \frac{\Delta}{N_c \cdot N} \forall(i,j,k)$, and $\gamma_{ik} = \frac{\Gamma}{N} \forall(i,k)$. One can verify that they satisfy \cref{eq:delta_ijk_sum,eq:gamma_ik_sum} with equality. 
\end{remark}

For the numerical examples in this work, we consider hyperplane state constraints and 2-norm control input constraints.
By substituting the expressions for $x_k, u_k$, these can be written in terms of the decision variables as: 
\begin{subequations}
\begin{align} \label{eq:halfplane-state}
&a_{j}^\top E_k [A \mu_0^i + B V + B L^i ( \mu_0^i - \mu_0^g)] +
\\
&\mathhl{F_{\normal}^{-1} (1 - \delta_{ijk})} \norm*{a_{j}^\top E_k(A + B L^i) \mathhl{(\Sigma_0^i)^{\frac{1}{2}}} }
  - b_{j} \leq 0, \forall (i,j,k) 
\nonumber
\\
\begin{split} \label{eq:2-norm-control}
& \norm*{v_k + L_k^i (\mu_0 - \mu_0^g)}
\\
& + \sqrt{\mathhl{F_{\chi_{n_u}^2}^{-1}}(1 -\gamma_{ik})} \norm*{L_k^i (\Sigma_0^i)^{\frac{1}{2}}}   \leq  u_{\max} , \forall (i,k)
\end{split}
\end{align}
\end{subequations}
combined with \cref{eq:delta_ijk_sum,eq:gamma_ik_sum}. 

Combining \cref{thm:1,thm:CC-deterministic}, with a known risk allocation, we have a deterministic convex optimization problem to solve \cref{pr:1} under chance constraints:
\begin{problem} \label{pr:final}
    Given the risk variables $\delta_{ijk}, \gamma_{ik}$ that satisfy \cref{eq:delta_ijk_sum,eq:gamma_ik_sum}, minimize \cref{eq:lqr-deterministic} (quadratic) or \cref{eq:minfuel-deterministic} (2-norm) subject to the terminal constraints \cref{eq:terminal-mean-final,eq:terminal-cov-GMM-final} and chance constraints \cref{eq:halfplane-state,eq:2-norm-control} \highlight{with respect to the variables $(u_k)_{k=0:N-1}$, $(L^i)_{i=1:K}$.}
\end{problem}
\subsection{Risk Allocation}%
\if\shortOrFull1 
To enhance the optimality of the solution from solving \cref{pr:final}, an iterative risk allocation (IRA) algorithm \cite{ono_iterative_2008} is a common choice.
Inspired by \cite{ono_iterative_2008,boone_non-gaussian_2022}, we develop a new IRA algorithm that simultaneously accounts for affine and 2-norm constraints and allocates the risk between nodes, constraints, and kernels.
At each iteration, \cref{pr:final} is solved. Then, for inactive constraints, their corresponding risk variables are decreased. The affine and 2-norm constraint risks are respectively updated as $\delta_{ijk} {\gets} \beta \delta_{ijk} + (1{-}\beta)
   (1 - \Phi_{\normal} \left\{ (b_j - a_j^\top \mu_k^i) /\sqrt{a_j^\top \Sigma_k^i a_j} \right\}) $, and $\gamma_{ik} \gets \beta \gamma_{ik} + (1{-}\beta)
    (1 - \Phi_{\chi^2}\{  (y_{\max} - \norm{\mu_{k}^i} )^2/\|\chol{\Sigma_{k}^i}\|^2 \}  )$, 
where $\beta$ is an algorithm parameter such that $0{<}\beta{<}1$. By inverting \cref{eq:CC_deterministic-affine,eq:control-norm-final}, one can verify that the risk will always be decreased for inactive constraints. The residual risks that occur from this risk reduction are calculated by
$\delta_{\mathrm{res}} {\gets} \Delta{-}\sum_{i=1}^{K} \sum_{j=1}^{N_c} \sum_{k=1}^N\alpha_i\delta_{ijk},    \gamma_{\mathrm{res}} \gets \Gamma {-} \sum_{i=1}^{K} \sum_{k=1}^N\alpha_i \gamma_{ik}
$. The residual risks are allocated to the active constraints. The allocation scheme differs depending on whether all Gaussian kernels have active constraints. For example, for affine constraints, if all the kernels are active, the update is simply $\delta_{ijk} \gets \delta_{ijk} +\delta_{\mathrm{res}} /N_{\mathrm{active,total}}^x$, where $N_{\mathrm{active,total}}^x$ is the total number of active hyperplane constraints. If there are inactive kernels, $\delta_{ijk} \gets \delta_{ijk} + \delta_{\mathrm{res}} / (\alpha_i\cdot N_{\mathrm{active},i}^x \cdot M_{\mathrm{active}}^x)$, where $M_{\mathrm{active}}^x$ is the number of active kernels and $N_{\mathrm{active},i}^x$ is the number of active constraints for kernel $i$. Intuitively, the residual risk is first distributed among all kernels by dividing by $M_{\mathrm{active}}^x$, then, for each kernel, we evenly distribute this by dividing by $\alpha_i\cdot N_{\mathrm{active},i}^x$. 
The update is similar for 2-norm constraints. This process is repeated until the improvement in cost is less than a tolerance $\epsilon$.  
The complete pseudocode is in \cite{kumagai_chance-constrained_2024}.

\fi
\if\shortOrFull2 
To enhance the optimality of the solution from solving \cref{pr:final}, an iterative risk allocation (IRA) algorithm \cite{ono_iterative_2008} is a common choice. Inspired by \cite{ono_iterative_2008,boone_non-gaussian_2022}, we develop a new IRA algorithm that simultaneously accounts for affine and 2-norm constraints, as well as allocation between nodes, constraints, and kernels. 
The pseudocode for the algorithm is shown in \cref{alg:ira}. $\Phi_{\normal}, \Phi_{\chi^2}$ represent the cumulative distribution function of the standard normal and chi-squared inverse distributions. $\indicator$ is the indicator function. The key enhancement to \cite{boone_non-gaussian_2022} is the consideration of 2-norm chance constraints and the difference in the risk update algorithm, which is a result of considering the \textit{weighted risk}, as termed in \cite{boone_non-gaussian_2022}.
By considering the \textit{weighted risk} instead of the \textit{unweighted risk}, the update procedure does not require any additional constraints on the risk variables. Since \cite{boone_non-gaussian_2022} considered the \textit{unweighted risk}, at each step of the algorithm, the risks must be updated while ensuring that they are not smaller than the corresponding mixture weights $\alpha_i$. 

Lines 10 and 11 count the number of kernels that have at least one active constraint for each type of constraint. 
Lines 12-15 decrease the risk allocated to the inactive constraints. By inverting \cref{eq:CC_deterministic-affine,eq:control-norm-final}, one can verify that the risk will always be decreased for inactive constraints. 
Lines 18-23 allocate the residual risk, created from decreasing the risk for inactive constraints, evenly among the active constraints. 
When $M_{\mathrm{active}}^x {=} K$, i.e. for all $i$ there is at least one active constraint, then, we simply divide the residual risk by the number of active constraints to get $\Delta {=} \sum_{i=1}^{K}\sum_{j=1}^{N_c}\sum_{k=1}^N \delta_{ijk}$. When $M_{\mathrm{active}}^x {<} K$, we want
\begin{equation}
    \delta_{\mathrm{res}} = \sum_{i\in I_{\mathrm{active}}} \sum_{j=1}^{N_c} \sum_{k=1}^N \alpha_i \delta_{ijk}
\end{equation}
where $I_{\mathrm{active}}$ is the set of active kernels. We first evenly distribute $\delta_{\mathrm{res}}$ among all active kernels by dividing by $M_{\mathrm{active}}^x$, then, for each kernel, we evenly distribute this by dividing by $\alpha_i\cdot N_{\mathrm{active},i}^x$.  

\begin{algorithm}[h]
\caption{Modified IRA-GMM Algorithm} \label{alg:ira}
\begin{algorithmic}[1]
\Require Convergence tolerance $\epsilon$, update parameter $\beta$ $(0<\beta<1)$
\State  $\forall(i,j,k) \quad \delta_{ijk} \gets \Delta /\left(N \cdot N_c\right)$
\State $\forall(i,k) \quad \gamma_{ik} \gets \Gamma / N$
\While {$\left|J^*-J_{\text {prev}}^*\right| > \epsilon$} 
    \State$J_{\text{prev}}^* \gets J^*$
\State Solve \cref{pr:final} with $\delta_{ijk}, \gamma_{ik}$
\State $N_{\mathrm{active,total}}^x, N_{\mathrm{active,total}}^u \gets$ total number of active constraints for state and control constraints
\If{ ($N_{\mathrm{active,total}}^x=N_{\mathrm{active,total}}^u = 0$) or ($N_{\mathrm{active,total}}^x=N \cdot N_c$ and $N_{\mathrm{active,total}}^u = N$) } {
    \textbf{break}}
\EndIf
\ForAll {$i$}
    \State $N_{\mathrm{active},i}^x, N_{\mathrm{active},i}^u \gets$ Number of active state (control)  constraints for $i$-th kernel
\EndFor
\State $M_{\mathrm{active}}^x = \sum_{i=1}^{K} \indicator_{N_{\mathrm{active},i}^x > 0}$
\State $M_{\mathrm{active}}^u =  \sum_{i=1}^{K} \indicator_{N_{\mathrm{active},i}^u > 0}$

\ForAll{ $i$ such that $j$-th hyperplane constraint is inactive for $i$-th kernel at time $t_k$}
   \State  $\delta_{ijk} \gets \beta \delta_{ijk} + (1-\beta)
   (1 - \Phi_{\normal} \left[\frac{b_j - a_j^\top \mu_k^i}{\sqrt{a_j^\top \Sigma_k^i a_j}} \right])$
\EndFor
\ForAll{ $i$ such that 2-norm constraint is inactive for $i$-th kernel at time $t_k$}
   \State  $\gamma_{ik} \gets \beta \gamma_{ik} + (1-\beta)
    (1 - \Phi_{\chi^2}\left[ \frac{ (y_{\max} - \norm{\mu_{k}^i} )^2}{\norm{\chol{\Sigma_{k}^i}}^2} \right]  )$
\EndFor

\State $\delta_{\mathrm{res}} \gets \Delta-\sum_{i=1}^{K} \sum_{j=1}^{N_c} \sum_{k=1}^N\alpha_i\delta_{ijk}$
\State $\gamma_{\mathrm{res}} \gets \Gamma - \sum_{i=1}^{K} \sum_{k=1}^N\alpha_i \gamma_{ik}$

\ForAll{$i$ such that $j$-th hyperplane constraint is active for $i$-th kernel at time $t_k$}
    \If{$M_{\mathrm{active}}^x = K$}
        \\
        {$\delta_{ijk} \gets \delta_{ijk} +\delta_{\mathrm{res}} /N_{\mathrm{active,total}}^x$}
    \Else {$\delta_{ijk} \gets \delta_{ijk} + \delta_{\mathrm{res}} / (\alpha_i\cdot N_{\mathrm{active},i}^x \cdot M_{\mathrm{active}}^x)$}
    \EndIf
\EndFor
\ForAll{$i$ such that 2-norm constraint is active for $i$-th kernel at time $t_k$}
    \If {$M_{\mathrm{active}}^u = K$}\\
        {$\gamma_{ik} \gets \delta_{ik} +\gamma_{\mathrm{res}} /N_{\mathrm{active,total}}^u$}
    \Else {$\gamma_{ik} \gets \gamma_{ik} + \gamma_{\mathrm{res}} / (\alpha_i\cdot N_{\mathrm{active},i}^u \cdot M_{\mathrm{active}}^u)$}
    \EndIf
\EndFor
\EndWhile
\end{algorithmic}
\end{algorithm}

\fi
\subsection{Discussion and Comparison with a Related Work} \label{sec:discussion}
\highlight{
Here, we remark on two key differences between our work and a related work \cite{balci_density_2023}. 
}

\highlight{
The first difference is the deterministic feedforward policy in \cref{eq:policy}.
In contrast to a probabilistic feedforward policy in \cite{balci_density_2023}, this provides a \textit{nominal} trajectory $\hat{x}$, defined as
\begin{equation}
    \hat{x}_{k+1} = A_k \hat{x}_k + B_k v_k, \quad \hat{x}_0 = \mu_0^g
\end{equation}
i.e. propagation from the initial mean with the feedforward control. 
The availability of such nominal trajectories permits a straightforward extension of the proposed approach to more complex problems. For instance, consider a nonlinear GM steering problem. Such a problem typically requires sequential solution methods, which detect convergence by evaluating the difference between the approximate solution based on the linearized system about a reference trajectory and its nonlinear response. The deterministic feedforward term makes it straightforward to calculate the nonlinear response deterministically. On the other hand, without such a nominal trajectory, the only analogous concept may be the mean trajectory, which, however, would require nonlinear uncertainty quantification, significantly increasing the computational demand. 
}

\highlight{
Another key difference lies in the fundamental philosophy of the problem formulation, which affects the optimality and computational efficiency in solving chance-constrained (CC) problems. \cite{balci_density_2023} decomposes the overall GM steering problem into multiple separate covariance steering problems (CSPs); this approach is less appealing for CC problems because, unlike unconstrained CSPs (which \cite{balci_density_2023} assumes), CC CSPs have no closed-form solution and need convex programming for a numerical solution \cite{okamoto_optimal_2018}. On the other hand, the proposed approach formulates the entire problem into a single convex programming under chance constraints. 
The decomposing approach can also induce conservativeness in the constraints. For instance, a simple approach to decompose \cref{eq:lambda_max<1} may be 
\begin{equation}
    \maxeig{\Sigma_f^{-\frac{1}{2}} \Sigma _N^i (\Sigma_f^{-\frac{1}{2}})^\top} \leq 1 \quad \forall i 
\end{equation}
which implies \cref{eq:lambda_max<1};  however, the opposite is false and hence suboptimal. On the other hand, \cref{prop:terminal-numerical} is \textit{exact} due to the fundamentally different problem formulation. 
}

\highlight{
Nevertheless, when considering non-constrained general GMM-to-GMM steering, \cite{balci_density_2023} is a powerful framework.
}
\section{Numerical Simulations} \label{sec:numerical}
We validate the proposed method with an example. Consider the following system with $N=20$:
\begin{equation*}
A_k = \smqty[1 & 0 & \Delta t & 0 \\
0 & 1 & 0 & \Delta t \\
0 & 0 & 1 & 0 \\
0 & 0 & 0 & 1], \ B_k=\smqty[
\Delta t^2 / 2 & 0 \\
0 & \Delta t^2 / 2 \\
\Delta t & 0 \\
0 & \Delta t
], \ \Delta t = 0.2, \quad \forall k
\end{equation*}%
The initial distribution is a Gaussian mixture with $K {=} 3$ kernels, with weights $(\alpha_1, \alpha_2, \alpha_3) = (0.3,0.4, 0.3)$, means $\mu_0^{(1)} = [5,{-1},5, 0]^\top$, $\mu_0^{(2)} = [3.5, 0.5, 8, 0]^\top$, $\mu_0^{(3)} = [4, -0.5, 7, 0]^\top$, and covariances $\Sigma_0^{(1)} = \Sigma_0^{(2)} = \Sigma_0^{(3)} = \diag{0.05,0.05,0.01,0.01}$. $N_c {=} 2$ affine state constraints are chosen as $a_1 = [1.3, {-1}, 0, 0]^\top, a_2 = [{-1}, 1, 0, 0]^\top, b_1 {=} 11, b_2 {=} {-1}$, with a joint violation probability of $\Delta {=}0.005$. The 2-norm constraint on control input is $u_{\max} {=} 6.5$, with a violation probability of $\Gamma {=} 0.005$. The target terminal distribution is $\mu_f = [8, 5.5, 0,  0]^\top, \Sigma_f = \diag{0.05,0.05,0.01,0.01}$. We choose the quadratic cost with $Q_k {=} 0, R_k {=} I$ for all $k$. 

All convex problems are solved using YALMIP \cite{lofberg_yalmip_2004} and MOSEK \cite{mosek_aps_mosek_2023}. First, we solve the problem without state chance constraints. 
\cref{fig:unconstrained} shows the problem setting and 1000 Monte Carlo sample trajectories. The samples are successfully steered to the target distribution. 
\cref{fig:constrained} shows the results when we impose state constraints \highlight{with URA}. Although the magnitude of the feedforward \highlight{control} is similar to the state-unconstrained case, the variance in magnitude is much larger under \highlight{state} constraints. 
\highlight{The state-unconstrained (constrained) case takes $\approx$1.5 (2.5) seconds on a standard laptop.}
\cref{fig:density} shows the theoretical \highlight{state} density evolution for the state-constrained case. The density \highlight{matches the Monte Carlo results from \cref{fig:constrained} well.}

\begin{figure}[t]
    \centering%
    \vspace{-5pt}%
    \begin{subfigure}[c]{0.3\linewidth}
        \includegraphics[width=\textwidth]{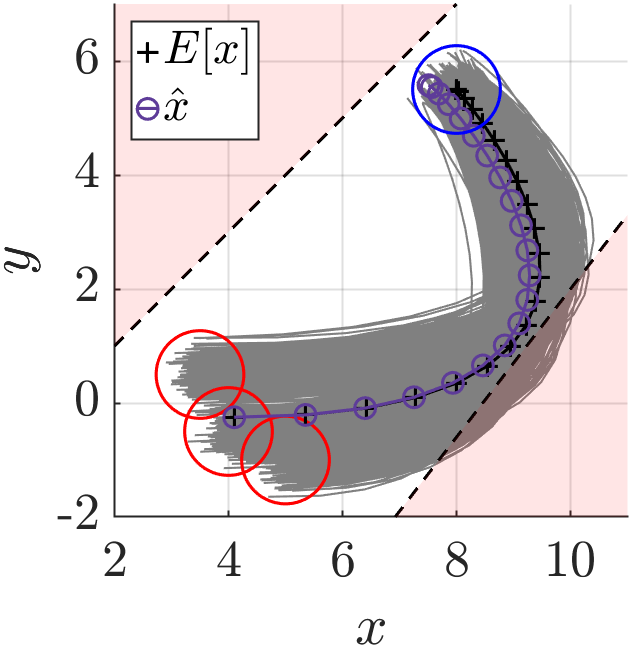}
    \end{subfigure}%
    \begin{subfigure}[c]{0.5\linewidth}
        \includegraphics[width=\textwidth]{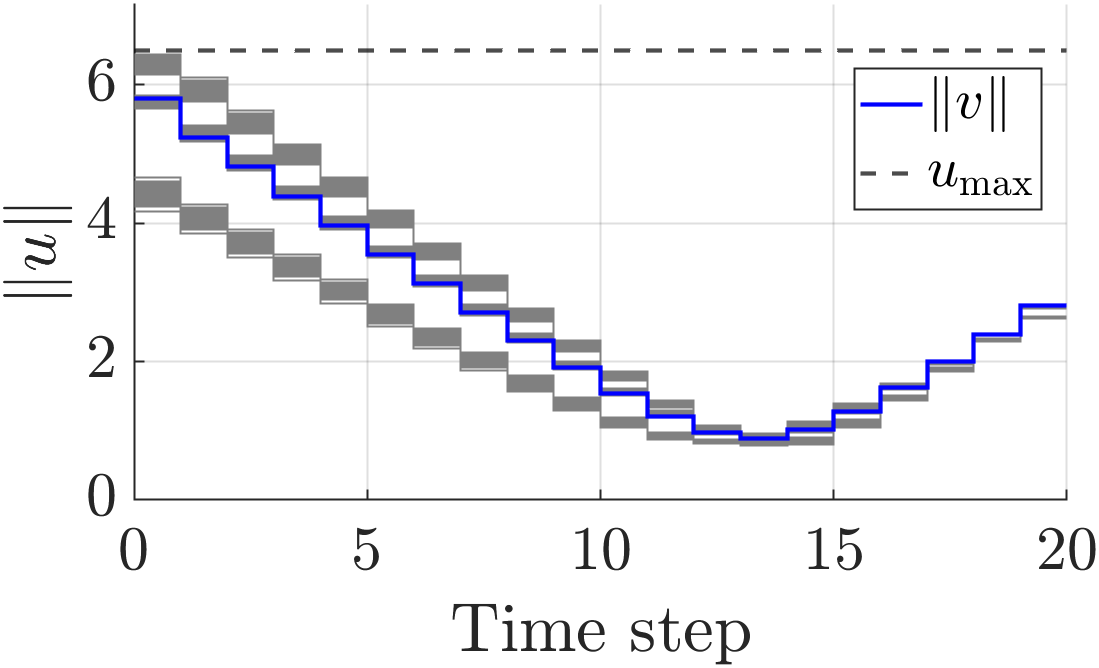}%
    \end{subfigure}
    \vspace{-5pt}
    \caption{\highlight{Monte Carlo with control chance constraints}}
    \label{fig:unconstrained}
    
    \begin{subfigure}[c]{0.3\linewidth}
    \includegraphics[width=\textwidth]{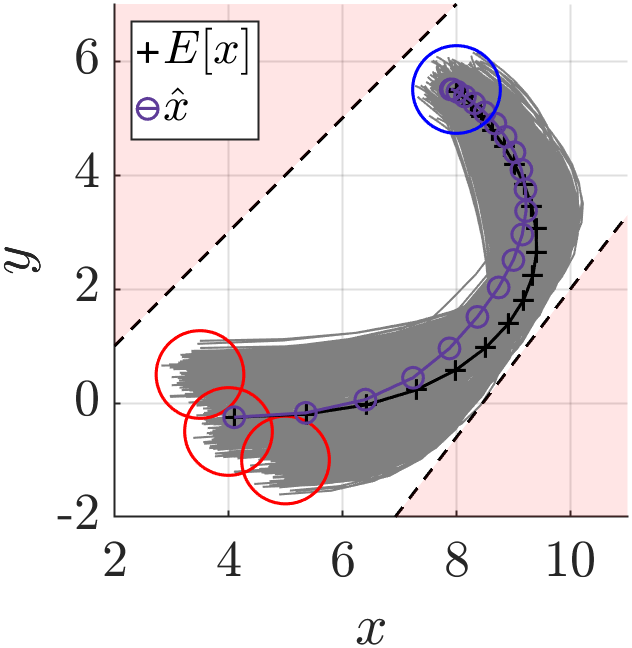}
    \end{subfigure}%
    \begin{subfigure}[c]{0.5\linewidth}
        \includegraphics[width=\textwidth]{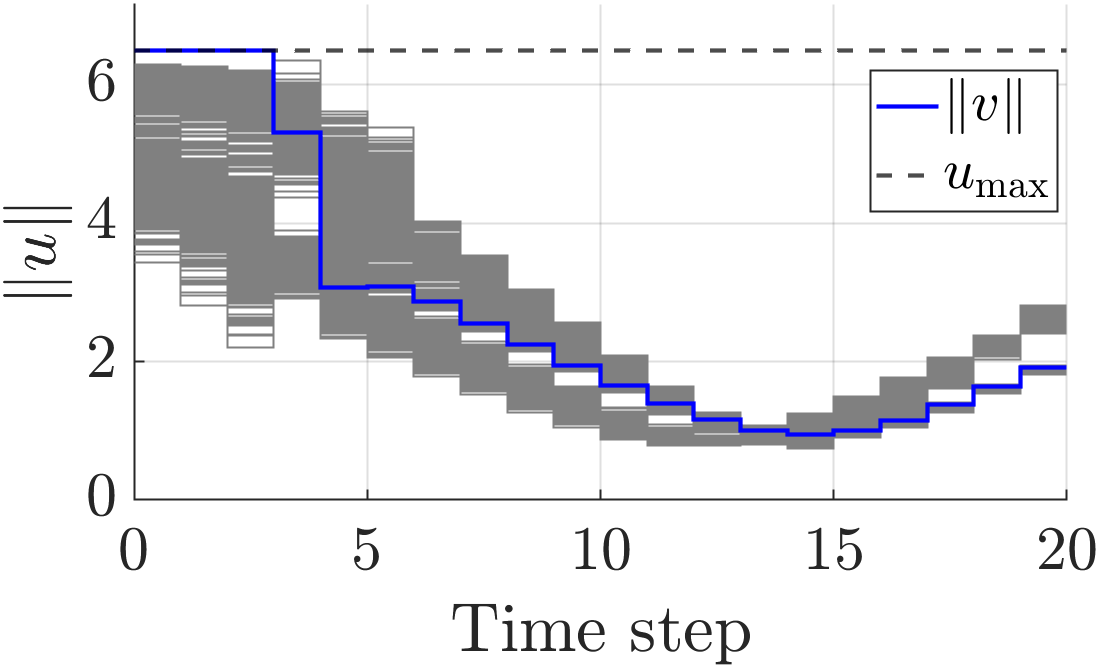}  
    \end{subfigure}%
    \vspace{-5pt}
    \caption{\highlight{Monte Carlo with state and control chance constraints}}
    \label{fig:constrained}
\end{figure}%

\begin{figure}[t]
    \centering
    \includegraphics[width=0.94\linewidth]{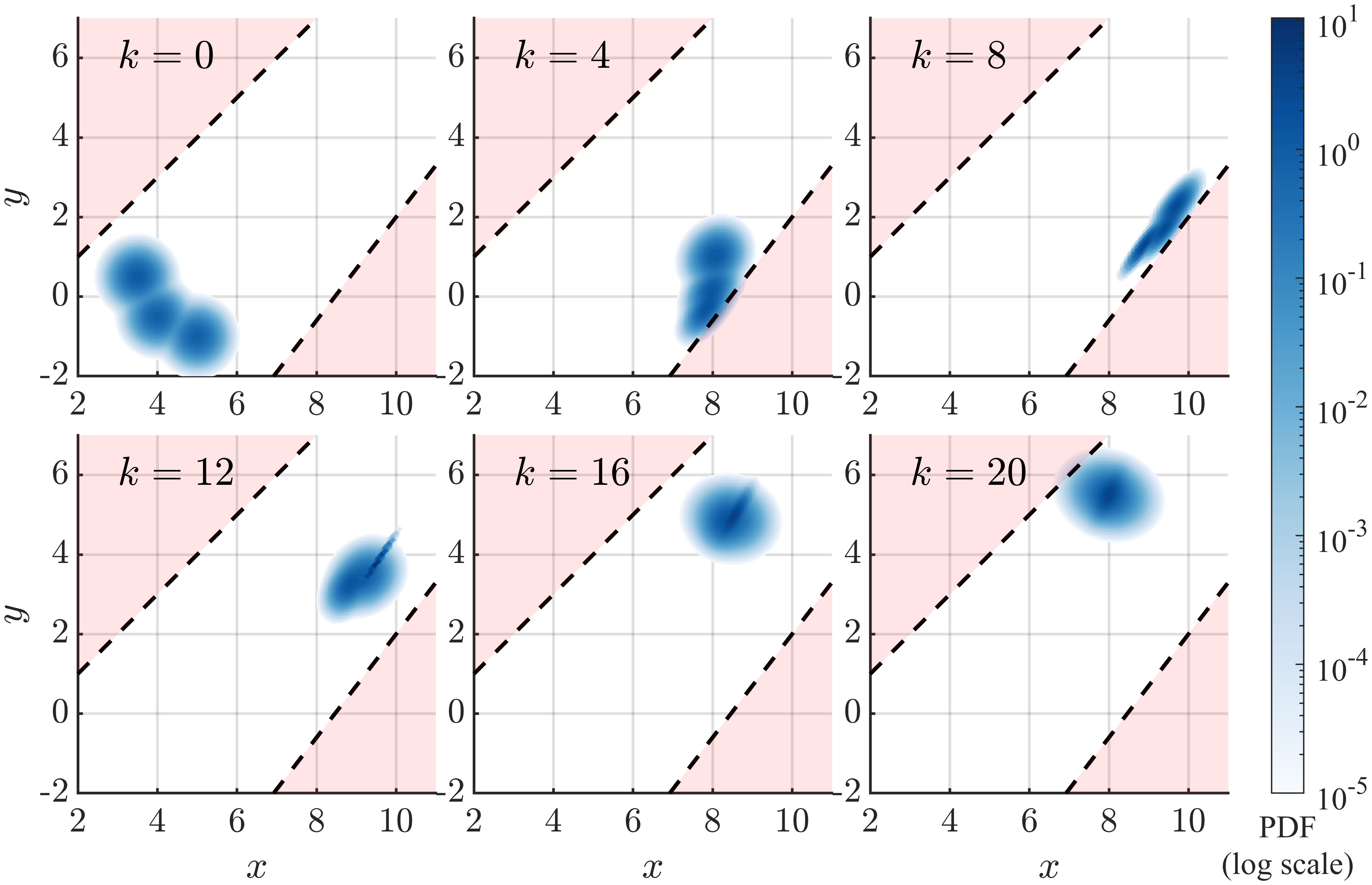}
    \vspace{-5pt}
    \caption{Evolution of \highlight{marginal} density for the path-constrained example.%
    }%
     \label{fig:density}
\end{figure}%

Next, we refine the solution to the state-constrained problem using the proposed IRA algorithm\highlight{, with $\epsilon {=} 10^{-2}$ and $\beta {=} 0.7$}. The algorithm converges after 13 iterations\highlight{, providing $\approx$5\% cost improvement}. 
\cref{fig:traj-comparison} compares the Monte Carlo trajectories with and without IRA. The mean\highlight{, nominal,} and \highlight{dispersed states} of the IRA-refined solution approach the halfplane closer.
\begin{figure}[t]
    \centering
    \begin{subfigure}[c]{0.4\linewidth}
        \includegraphics[width=\textwidth]{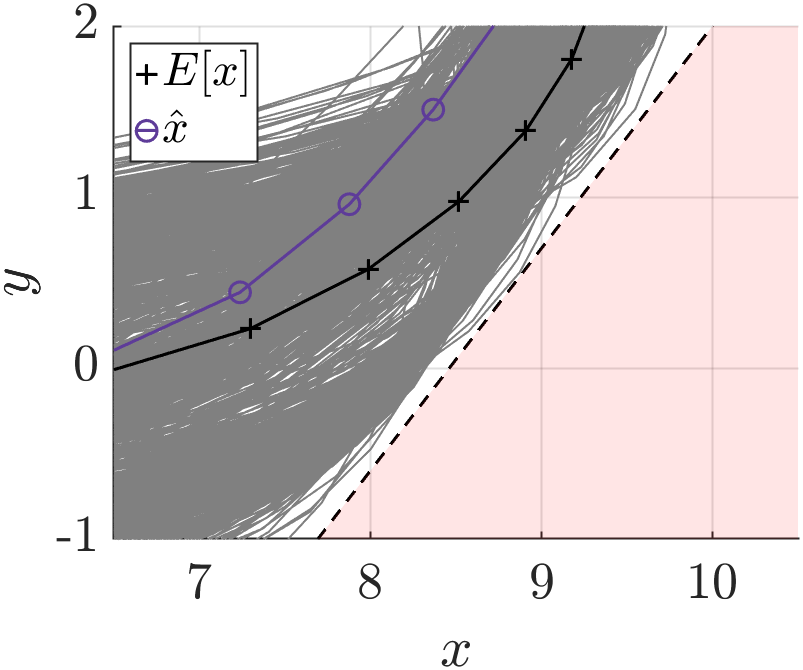}%
        \vspace{-5pt}
        \caption{without IRA \highlight{(URA)}} 
    \end{subfigure}%
    \begin{subfigure}[c]{0.4\linewidth}
        \includegraphics[width=\textwidth]{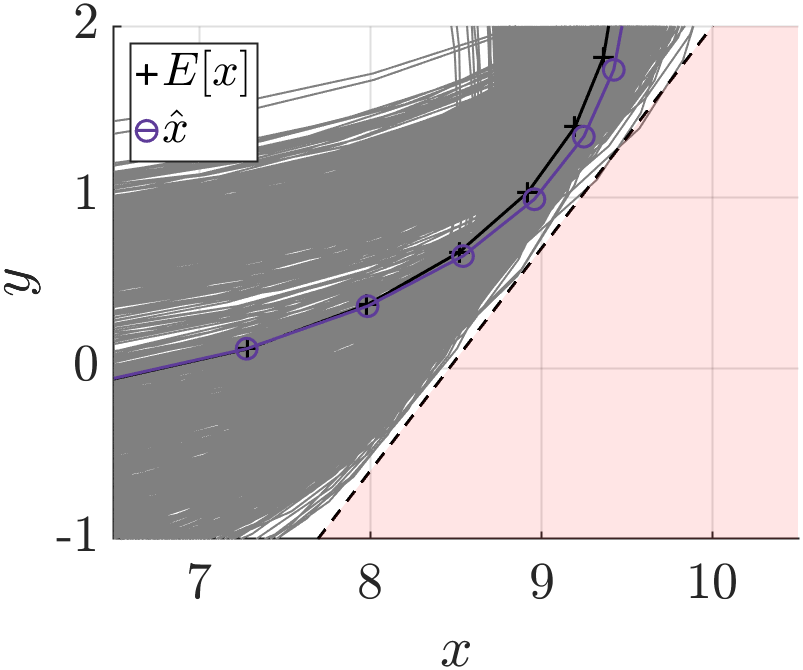}%
        \vspace{-5pt}
        \caption{with IRA}
    \end{subfigure}%
    \vspace{-5pt}%
    \caption{Comparison of the trajectory of Monte Carlo samples}%
    \vspace{-5pt}
    \label{fig:traj-comparison}
\end{figure}%
\if\shortOrFull1 

Finally, we compare the 2-norm cost for the same settings. An interesting observation is that the control magnitude is `bang-bang'-like but also shows the multi-modal distribution derived in \cref{prop:control-distribution}. The multi-modal structure is especially observable in the switching phase between maximal and minimal inputs.
Due to space restrictions, figures for the 2-norm cost are included in \cite{kumagai_chance-constrained_2024}.
\fi
\if\shortOrFull2 
The history of cost is shown in \cref{fig:IRA-history}. We see that the value decreases monotonically. 
\begin{figure}[tb]
    \centering
    \includegraphics[width=0.6\linewidth]{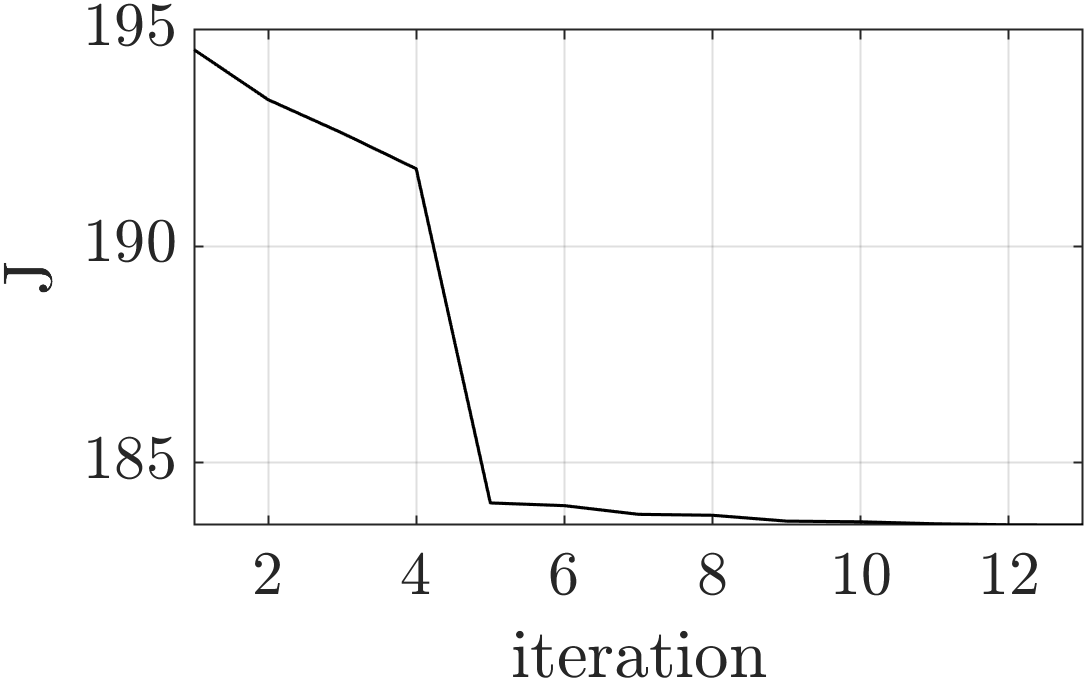}
    \caption{History of cost for the IRA algorithm}
    \label{fig:IRA-history}
\end{figure}

Finally, we compare the 2-norm cost for the same settings. The trajectory approaches the target more directly. The terminal covariance constraint is not active. The control magnitude is `bang-bang'-like but also shows the multi-modal distribution derived in \cref{prop:control-distribution}. The multi-modal structure is especially observable in the switching phase between maximal and minimal inputs.%
\begin{figure}[t]
    \centering
    \begin{subfigure}[c]{0.4\linewidth}
    \includegraphics[width=\textwidth]{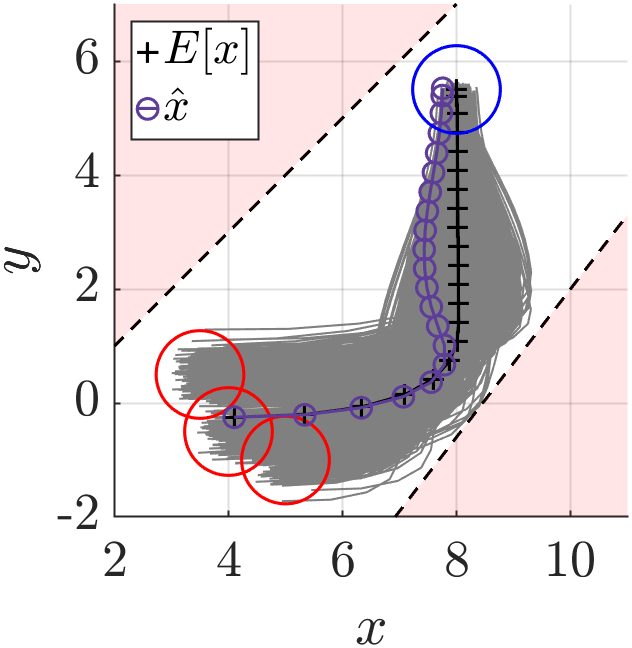} 
    \end{subfigure}%
    \begin{subfigure}[c]{0.5\linewidth}
        \includegraphics[width=\textwidth]{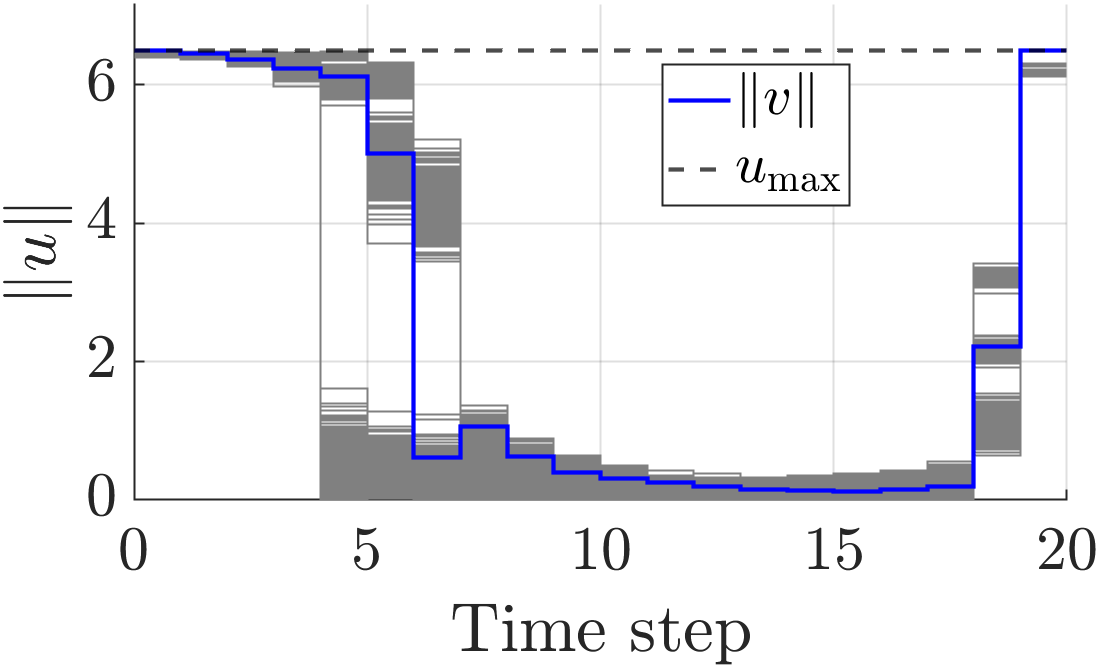}  
    \end{subfigure}
    \caption{(a) Trajectory and (b) control magnitude for 2-norm cost.}
    \label{fig:minfuel}
\end{figure}
\fi
\section{Conclusion}
We have addressed the problem of Gaussian mixture-to-Gaussian distribution steering under chance constraints. By using a probabilistically chosen affine control policy, the state and control distributions throughout the time horizon are fully characterized by Gaussian mixture models. The original problem is converted to a single \highlight{convex optimization} problem, by deriving the deterministic formulations for cost, terminal distributional constraint, and affine/2-norm chance constraints. We also modify the risk allocation algorithm for reduced conservativeness. 


\bibliographystyle{ieeetr}
\bibliography{CDC2024}
\end{document}